\documentclass[12pt]{amsart}

\usepackage{amssymb,stmaryrd}
\usepackage{booktabs}
\usepackage{cases}
\usepackage{array} 
\usepackage{float} 
\usepackage{rotating}

\usepackage{tikz}
\usetikzlibrary{positioning}
\usepackage{rank-2-roots}

\usepackage{color}
\usepackage{etoolbox} 
\usepackage{expl3}
\usepackage{pgfkeys}
\usepackage{pgfopts}
\usepackage{xparse}
\usepackage{xstring}

\usepackage[foot]{amsaddr}
\makeatletter
\renewcommand{\email}[2][]{%
  \ifx\emails\@empty\relax\else{\g@addto@macro\emails{,\space}}\fi%
  \@ifnotempty{#1}{\g@addto@macro\emails{\textrm{(#1)}\space}}%
  \g@addto@macro\emails{#2}%
}
\makeatother
\usepackage{mathtools}
\usepackage{enumerate,pdflscape}
\usepackage{multirow}
\usepackage{color}
  \usepackage[titletoc]{appendix}
\newcommand{\stt}{\mathfrak{s}_{T, \mathfrak{t}}}
\newcommand{\sttt}{\mathfrak{s}_{[T\cup-T], \mathfrak{h}}}

\makeatletter
\newcommand{\svast}{\bBigg@{3}}
\newcommand{\vast}{\bBigg@{4}}
\newcommand{\Vast}{\bBigg@{5}}
\makeatother

\newcommand{\wc}{[ T \cup -T]}
\newcommand{\cw}{[T \cup -T ]}

 \newcommand{\ad}{\text{ad}}

  \theoremstyle{definition}
  \newtheorem{definition}{Definition}[section]

   \theoremstyle{plain}

  \newtheorem{lemma}[definition]{Lemma}
  \newtheorem{proposition}[definition]{Proposition}
  \newtheorem{theorem}[definition]{Theorem}
  \newtheorem{corollary}[definition]{Corollary}

    \theoremstyle{definition}

\newcommand\End{\text{End}}

\title[Narrow, wide, and $\lambda$-wide regular subalgebras]{Narrow, wide, and $\lambda$-wide regular  subalgebras of semisimple Lie algebras}

\begin{document}

\author[Andrew Douglas]{Andrew Douglas$^{1,2}$}
\address[]{$^1$Department of Mathematics, New York City College of Technology, City University of New York, Brooklyn, NY, USA.}
\address[]{$^2$Ph.D. Programs in Mathematics and Physics, CUNY Graduate Center, City University of New York, New York, NY, USA.}

\author[Joe Repka]{Joe Repka$^3$}
\address{$^3$Department of Mathematics, University of Toronto, Toronto, ON,  Canada.}

\date{\today}


\begin{abstract}
A subalgebra  of a semisimple Lie algebra  is {\it wide} if every simple module of the semisimple Lie algebra remains indecomposable when restricted to the subalgebra. From a finer viewpoint, a subalgebra  is  $\mathit{\lambda}${\it -wide} if the simple module  of a semisimple Lie algebra of highest weight $\lambda$ remains indecomposable when restricted to the subalgebra. 
A subalgebra  is {\it narrow} if the restriction of all non-trivial simple modules  to the subalgebra have 
proper 
 decompositions.
We determine necessary and sufficient conditions for  regular  subalgebras  of  semisimple Lie algebras to be $\lambda$-wide.  As a natural consequence, we  establish necessary and sufficient conditions for regular subalgebras to be wide, a result which has already been established by Panyushev  for essentially all regular solvable subalgebras \cite{panyu}.
Next, we show that establishing whether or not a regular subalgebra  of a simple Lie algebra is wide does not require consideration of all simple modules. It is necessary and sufficient to only consider the adjoint representation.
Finally, we show that a regular subalgebra of the special linear algebra $\mathfrak{sl}_{n+1}$  is either  narrow or wide; this property  does not hold for non-regular subalgebras of $\mathfrak{sl}_{n+1}$.
\end{abstract}

\maketitle

\section{Introduction}\label{intro}

A subalgebra  of a semisimple Lie algebra  is {\it wide} if every simple module of the semisimple Lie algebra remains indecomposable when restricted to the subalgebra. 
Notable research has examined wide subalgebras of semisimple Lie algebras.  In \cite{dp},  for instance, Douglas and Premat  showed that a regular subalgebra, isomorphic to the Euclidean algebra $\mathfrak{e}(2)$, is wide in the special linear  algebra $\mathfrak{sl}_3$. 

Douglas, Repka, and Joseph  \cite{dwr2} showed  that all subalgebras isomorphic to the  solvable, $4$-dimensional Diamond Lie algebra are wide in both 
$\mathfrak{sl}_3$, and  the symplectic  algebra $\mathfrak{sp}_4$. 
Douglas and Repka showed that all subalgebras of $\mathfrak{sp}_4$ isomorphic to the Euclidean algebra $\mathfrak{e}(3)$ are wide \cite{drconf}. 

Casati proved that certain $n$-dimensional abelian subalgebras of  $\mathfrak{sl}_{n+1}$ are wide \cite{casati}, and that the Levi decomposable algebra $\mathfrak{sl}_{n+1}  \inplus \mathbb{C}^{n+1}$ is wide in $\mathfrak{sl}_{n+2}$ \cite{casatib}. More generally, Panyushev \cite{panyu}--who seems to have coined the term ``wide"--established significant and beautiful results including describing 
necessary and sufficient conditions for essentially all  regular solvable subalgebras of a semisimple Lie algebra to be wide.

In this article, we first shift our focus from wide subalgebas to   $\lambda$-wide subalgebras.  A subalgebra  is  $\mathit{\lambda}${\it -wide} if the simple module  of a semisimple Lie algebra of highest weight $\lambda$ remains indecomposable when restricted to the subalgebra.
In particular, we determine necessary 
and sufficient conditions for  regular  subalgebras  of a semisimple Lie algebra to be $\lambda$-wide.  
As a consequence, we also derive necessary and sufficient conditions for regular subalgebras to be wide, a result which has already been established by Panyushev for  essentially all regular solvable subalgebras \cite{panyu}.   
Next, we show that establishing whether or not a regular subalgebra  of a simple Lie algebra is wide does not require consideration of all simple modules.  It is sufficient to consider only the adjoint representation.

Finally we examine narrow subalgebras. That is, a subalgebra whose restriction  to any non-trivial simple module of the semisimple Lie algebra has a proper decomposition.
A non-regular subalgebra of $\mathfrak{sl}_{n+1}$  may be neither narrow nor wide. Indeed, in  \cite{dr1}, Douglas and Repka showed that  $\mathfrak{e}(3)$, as a non-regular subalgebra, 
is neither narrow nor wide  in $\mathfrak{sl}_4$; When restricted to this non-regular subalgebra, some simple modules remain indecomposable, and others do not. 

We show that a regular subalgebra of $\mathfrak{sl}_{n+1}$, however, must be either narrow or wide. That is, either all non-trivial 
simple modules of $\mathfrak{sl}_{n+1}$  have  a proper decomposition when restricted to the regular subalgebra, or all simple modules remain indecomposable upon restriction to the regular subalgebra. 
We use  bases of simple $\mathfrak{sl}_{n+1}$-modules created by Feigin,  Fourier, and Littelmann \cite{feigin}.

The results of this article include regular solvable, semisimple, and Levi decomposable subalgebras. 
 Hence, all regular subalgebras are included  (Levi's Theorem [\cite{levi}, Chapter II, Section $2$]). 

The article is organized as follows. Section \ref{regular} contains necessary background information, and establishes notation. In Section \ref{lwidesec}, wide, and $\lambda$-wide subalgebras are examined. 
In  Section \ref{narrowwidesection}, we prove that regular subalgebras of $\mathfrak{sl}_{n+1}$ are either narrow or wide. 

Note that all Lie algebras and modules in this article are over the complex numbers, and  finite-dimensional.

\section{Background and notation}\label{regular}

In this section, we review relevant background on semsimple Lie algebras, their modules, closed subsets of root systems, and regular subalgebras of semisimple Lie algebras. We begin with semisimple Lie algebras and their modules, which largely follows \cite{humphreys}. 

\subsection{Semisimple Lie algebras and their modules}

Throughout this paper,  $\mathfrak{g}$ will denote a semisimple Lie algebra, $\mathfrak{h}$ a fixed Cartan subalgebra of $\mathfrak{g}$, with corresponding root system $\Phi$, and Weyl group $\mathcal{W}$. For $\alpha \in \Phi$,  we denote $\mathfrak{g}_\alpha$  the corresponding root space.
The set of positive roots of $\Phi$ is denoted $\Phi^+$, and $\Delta =\{ \alpha_1,...,\alpha_n\}\subseteq\Phi^+$ is a base of $\Phi$. The {\it rank} of $\mathfrak{g}$ is $n$,  the number of simple roots in $\Delta$.
We denote the root lattice by $\mathcal{Q} =\oplus_{i=1}^n \mathbb{Z} \alpha_i$. 

We may naturally associate $\mathfrak{h}$ with its dual space $\mathfrak{h}^*$ via the Killing form $\kappa$. Specifically,  $\alpha \in \mathfrak{h}^*$ 
corresponds to the unique element $t_\alpha \in \mathfrak{h}$ such that $\alpha(h) = \kappa(t_\alpha, h)$, for all $h \in \mathfrak{h}$. We have the nondegenerate symmetric bilinear form  on $\mathfrak{h}^*$ given by
$(\alpha, \beta) \coloneqq \kappa(t_{\alpha}, t_{\beta} )$, and we may define  $\langle \alpha, \beta \rangle \coloneqq \frac{2(\alpha, \beta)}{(\beta, \beta)} =\alpha (h_{\beta})$, where 
$h_{\beta} \coloneqq \frac{2t_\beta}{\kappa(t_\beta, t_\beta)}=\frac{2t_\beta}{(\beta, \beta)}$.

 If $\alpha \in \Phi$, we may select a  nonzero $e_\alpha \in \mathfrak{g}_\alpha$, for which there 
 is a unique $e_{-\alpha} \in \mathfrak{g}_{-\alpha}$, such that $e_\alpha$, $e_{-\alpha}$, and $h_{\alpha} =[e_\alpha, e_{-\alpha}] \in \mathfrak{h}$ span a subalgebra 
isomorphic to $\mathfrak{sl}_2$, with commutation relations  $[h_\alpha, e_\alpha]=\alpha(h_\alpha) e_\alpha=2 e_\alpha$, and $[h_\alpha, e_{-\alpha}]=-2 e_{-\alpha}$. 
We refer to 
 $e_\alpha$,    $e_{-\alpha}$, and $h_{\alpha}$ as the $\mathfrak{sl}_2$-triple determined by $e_\alpha$.  Note that $h_{-\alpha} = - h_\alpha$.

The set of weights relative to the root system $\Phi$ is denoted $\Lambda$,  and $\Lambda^+$ is the set of all dominant weights with respect to $\Delta$. 
Let $\lambda_1,..., \lambda_n$ be the dual
basis to $\frac{2\alpha_i}{ (\alpha_i, \alpha_i)}$, where $ \alpha_i \in\Delta$ for all $i$. That is, $\frac{2(\lambda_i, \alpha_j)}{ (\alpha_j, \alpha_j)}=\delta_{ij}$. Each of $\lambda_1,..., \lambda_n$ is a dominant  weight, collectively called the {\it fundamental dominant weights} (relative to $\Delta$).

A dominant  weight $\lambda \in \Lambda^+$ may be written as $\lambda = m_1 \lambda_1+ \cdots +m_n \lambda_n$, where $m_i$ is a nonnegative integer for each $i$. For each dominant weight
$\lambda$, $V(\lambda)$ is  the simple $\mathfrak{g}$-module of highest weight $\lambda$. Fix a highest weight vector $v_\lambda \in V(\lambda)$, unique up to a scalar multiple.

For an arbitrary $\mathfrak{g}$-module $V$, let $\Pi(V)$ be the set of weights of $V$. Then, $V$
decomposes into weight spaces
\begin{equation}
 V = \bigoplus_{\mu \in \Pi(V) } V_\mu,
\end{equation}
where $V_\mu =\{ v\in V ~|~ h \cdot v = \mu(h) v, ~ \text{for all}~ h\in \mathfrak{h} \}$. Further, 
$\langle \mu, \alpha \rangle= \mu (h_{\alpha})\in \mathbb{Z}$, for each $\alpha \in \Phi$.
If $\mathfrak{s}$ is a subalgebra of $\mathfrak{g}$, define
\begin{equation}
V^{\mathfrak{s}} \coloneqq \{v\in V ~ | ~   \mathfrak{s} \cdot v =0\}.
\end{equation}

\subsection{Closed subsets of root systems and regular subalgebras}

A subset $T$ of the root system $\Phi$ is {\it closed} if for any
$x, y \in T$, $x+y \in \Phi$ implies $x+y \in T$. 
Any closed set $T$ is a disjoint union of its {\it symmetric} component $T^r =\{\alpha \in T | -\alpha \in T  \}$, 
and its {\it special} component $T^u=\{ \alpha \in T | -\alpha  \notin T \}$.

Closed subsets $T$ and $T'$ of a root system $\Phi$ are {\it conjugate} if there exists an element $w\in \mathcal{W}$ such that
$w(T)=T'$. The following lemma is well-known. 
\begin{lemma}\label{lem:deccl}\cite{sopkina}
$T^r$ is a closed root subsystem of $\Phi$. For any two roots $\alpha \in T^u$ and $\beta \in T$ such that $\alpha+\beta$ is a root, we have that  $\alpha+\beta \in T^u$. In particular, $T^u$ is closed.
\end{lemma}

Let $G$ be the adjoint group of $\mathfrak{g}$. That is, $G$ is the connected 
algebraic subgroup of $\mathrm{GL}(\mathfrak{g})$ with Lie algebra $\ad \mathfrak{g}$.
The Weyl group $\mathcal{W}$ of $\Phi$ naturally acts on the dual space $\mathfrak{h}^*$. Further, by identifying  $\mathfrak{h}$ and $\mathfrak{h}^*$ via the Killing form, the Weyl group
also acts on $\mathfrak{h}$.

Let $T\subseteq \Phi$ be a closed subset. Let $\mathfrak{t}$ be a subspace of $\mathfrak{h}$
containing $[\mathfrak{g}_\alpha, \mathfrak{g}_{-\alpha}]$ for each $\alpha \in T^r$. Then
\begin{equation}\label{reggg}
\mathfrak{s}_{T,\mathfrak{t}} = \mathfrak{t} \oplus \bigoplus_{\alpha \in T} \mathfrak{g}_\alpha  
\end{equation}
is a regular subalgebra of $\mathfrak{g}$.  Moreover, all regular subalgebras of $\mathfrak{g}$ normalized by $\mathfrak{h}$ arise in this manner. 
Further, any regular subalgebra of $\mathfrak{g}$ is conjugate under $G$ to a regular subalgebra normalized by $\mathfrak{h}$. Hence, 
we'll assume that any regular subalgebra is normalized by $\mathfrak{h}$, and thus in the form of Eq. \eqref{reggg}.

\begin{proposition}\label{wprop}[ \cite{dougdeg}, Proposition 5.1]
The regular subalgebras $\mathfrak{s}_{T_1, \mathfrak{t}_1}$ and  $\mathfrak{s}_{T_2, \mathfrak{t}_2}$ are conjugate under $G$ if and only if there is a $w \in \mathcal{W}$ with 
$w(T_1)=w(T_2)$ and $w( \mathfrak{t}_1)=w( \mathfrak{t}_2)$.
\end{proposition}

Let $S \subseteq \Phi$. The {\it closure} of $S$, denoted $[S]$, is the smallest closed subset of $\Phi$ containing $S$. 
The following result is straightforward to prove.
\begin{lemma}\label{semireg}
Let $T$ be a closed subset of $\Phi$. Then, $[T\cup -T]$ is a symmetric closed subset of $\Phi$ containing $T$. 
\end{lemma}

A regular solvable subalgebra is given by a subalgebra  of the form $\mathfrak{s}_{T,\mathfrak{t}}$, where $T$ is a special closed subset of $\Phi$ (possibly empty), and $\mathfrak{t}$ a subalgebra of $\mathfrak{h}$. 
A regular semisimple subalgebra is given by $\mathfrak{s}_{T, \mathfrak{t}}$ for some   symmetric closed subset $T$ (non-empty)  of $\Phi$, and subalgebra $\mathfrak{t}$ of $\mathfrak{h}$ generated by $[\mathfrak{g}_\alpha, \mathfrak{g}_{-\alpha}]$ for all $\alpha \in T\cap \Phi^+$. A regular Levi decomposable subalgebra is a non-semisimple regular Lie algebra $\mathfrak{s}_{T, \mathfrak{t}}$, 
such that $T$ is a closed subset, $T^r$ is the corresponding (non-empty) symmetric closed subset, and $T^u$ is the corresponding  special closed subset of $\Phi$, and  $\mathfrak{t}$ is  a subalgebra of $\mathfrak{h}$ containing $[\mathfrak{g}_\alpha, \mathfrak{g}_{-\alpha} ]$ for each $\alpha \in T^r$.

\section{Wide and $\lambda$-wide regular subalgebras}\label{lwidesec}

In this section, we 
suppose that $\lambda \in \Lambda^+$ and 
determine necessary 
and sufficient conditions for a regular  subalgebra  of a semisimple Lie algebra to be $\lambda$-wide (Theorem \ref{lwide}).  As a consequence of this theorem, 
we obtain necessary and sufficient conditions for regular  subalgebras  of a semisimple Lie algebra to be wide (Corollary \ref{lwideb}). Finally, we show that establishing whether or not a regular subalgebra of a simple Lie algebra is wide does not require consideration of all simple modules.  It is necessary and sufficient to consider only the adjoint representation (Corollary \ref{best}).

Before our main results, we present various lemmas.  
Lemmas \ref{gzero},  \ref{notpandn}, and \ref{nzero} have   similarity to lemmas in \cite{panyu}, but are specialized and generalized to the context of the present article.
We begin with a definition.

Let $T$ be  a closed subset of $\Phi$. Then, define an $\mathfrak{s}_{\wc, \mathfrak{h}}$-submodule of $V(\lambda)$: 
\begin{equation}
\cw \cdot \lambda \coloneqq \text{Span}\big\{ e_{-\beta_1} \cdots e_{-\beta_m}\cdot v_\lambda ~|~ -\beta_j \in \cw \cap \Phi^-\big\},
\end{equation}
where $j=1,...,m$, and such that $-\beta_1,...,$ and $-\beta_m$  
 are not necessarily distinct.  Note that we include $v_\lambda \in \cw \cdot \lambda$.
In addition, note that if  $e_{-\beta_1} \cdots e_{-\beta_m}\cdot v_\lambda  \neq 0$, then $\lambda -  \beta_1-  \cdots - \beta_m \in \Pi(V(\lambda))$.   %

\begin{lemma}\label{gzero}
Let $T$ be a  closed subset of $\Phi$, 
and $\mathfrak{t}$  a subalgebra of $\mathfrak{h}$
containing $[\mathfrak{g}_\alpha, \mathfrak{g}_{-\alpha}]$ for each $\alpha \in T^r$. Let 
$V$ be a $\mathfrak{g}$-module, and 
 $\mu \in \Pi (V)$, such that $V^{\mathfrak{s}_{T,\mathfrak{t}}}_{\mu} \neq \{0\}$. Then 
$\langle \mu, \beta \rangle \geq 0$, for all $\beta \in T$. 
And, if $\beta\in T$, and $-\beta \in T$, then $\langle\mu, \beta \rangle =0$. Further,
if both $V^{\mathfrak{s}_{T,\mathfrak{t}}}_{\mu} \neq \{0\}$, and $V^{\mathfrak{s}_{T,\mathfrak{t}}}_{-\mu} \neq \{0\}$, then $\langle\mu, \beta \rangle =0$, 
for all $\beta \in T$.  
\end{lemma}
\begin{proof}
The second and third conclusions follow from the first. Hence, we only address the first.

Consider a nonzero $v \in V^{\mathfrak{s}_{T,\mathfrak{t}}}_{\mu}$. Then,  $\mathfrak{g}_\beta \cdot v =0$ for all $\beta \in T$. 
Let $e_\beta \in \mathfrak{g}_\beta \setminus \{0\}$, with corresponding $\mathfrak{sl}_2$-triple $e_\beta$, $e_{-\beta}$, $h_\beta$.
Then, $e_\beta \cdot v =0$, so that $\mathfrak{sl}_2$-theory implies $h_\beta \cdot v = \mu(h_\beta) v$ with $\mu(h_\beta)=\langle \mu, \beta \rangle \geq 0$. 
\end{proof}

\begin{lemma}\label{wideclosure}
Let $S$ be a   closed subset of $\Phi$. Suppose $\beta_1$, $\beta_2$,...,$\beta_l \in S$, and $\beta_1+\beta_2+\cdots+\beta_l \in \Phi$. Then $\beta_1+\beta_2+\cdots +\beta_l \in S$. 
\end{lemma}
\begin{proof}
Let $\alpha=\beta_1+\cdots +\beta_l \in \Phi$.  By way of contradiction, suppose that  $\alpha  \notin S$. Then, in particular,
$\alpha \neq \beta_i$ for all $i$ with $1\leq i \leq l$. We also must have $l>2$, since we get an immediate contradiction due to the closure of $S$ otherwise.

 Suppose that $(\alpha, \beta_i) \leq 0$ for all $ 1\leq i \leq l$. Then,  $(\alpha, \alpha) =(\alpha, \beta_1)+\cdots+(\alpha, \beta_l) \leq 0$. Hence,
$(\alpha, \alpha)=0$, which implies $\alpha=0$. A contradiction. Thus, it must be the case that $(\alpha, \beta_{k_1}) >0$ for some $k_1$ with  $ 1\leq k_1 \leq l$. 
Since $\alpha \neq \beta_{k_1}$ (and $\alpha \neq -\beta_{k_1}$ since $(\alpha, \beta_{k_1}) >0$), we have $\alpha-\beta_{k_1}= \beta_1+\cdots + \beta_{k_1-1}+\beta_{k_1+1}+\cdots+\beta_l \in \Phi$  [\cite{humphreys}, Lemma 9.4].

If $\alpha-\beta_{k_1}\in S$,  then $\alpha=(\alpha-\beta_{k_1}) +\beta_{k_1} \in S$, by the closure of $S$, which contradicts our initial assumption that $\alpha \notin S$. Hence,
$\alpha-\beta_{k_1}\in  \Phi \setminus S$. In particular,  $\alpha-\beta_{k_1}\neq \beta_{i}$ for all $i$ with $1\leq i \leq l$.
Following the procedure above, we identify $k_2 \neq k_1$ such that 
$\alpha-\beta_{k_1} -\beta_{k_2} \in \Phi$. If $\alpha-\beta_{k_1} -\beta_{k_2} \in S$, then 
$\alpha-\beta_{k_1}=(\alpha-\beta_{k_1} -\beta_{k_2})+\beta_{k_2}\in S$, by the closure of $S$. But, this contradicts our finding that
$\alpha-\beta_{k_1} \in \Phi\setminus S$. Thus, we must have $\alpha-\beta_{k_1} -\beta_{k_2} \in \Phi\setminus S$.

We continue this procedure until we get $\alpha -\beta_{k_1} - \beta_{k_2} -\cdots - \beta_{k_{l-2}} \in  \Phi \setminus S$, which is
a sum of two elements of $S$.  However, this
 is an irreparable contradiction to the closure of $S$. Hence, it must be the case that $\alpha\in S$,  as required.
\end{proof}

\begin{lemma}\label{greatlemma}
Let $\mu$, $\nu \in \Pi(V(\lambda))$, $T$ a closed subset of $\Phi$, and $\wc \cdot \lambda =V(\lambda)$. If $\mu(h_\beta) = \nu(h_\beta)$ for all $\beta \in \wc$,
then $\mu =\nu$. 
\end{lemma}
\begin{proof}
Let $\alpha \in \Phi \setminus \cw$. We will show that $\mu(h_\alpha)=\nu(h_\alpha)$. 
We may assume $\alpha \in \Phi^+$, since $h_\alpha =-h_{-\alpha}$. We proceed in cases.

\vspace{2mm}

\noindent Case $1$. $\lambda(h_\alpha)\neq 0$: Then, considering the $\alpha$-string through $\lambda$, and that $\alpha \in \Phi^+$, we must have $\lambda -\alpha \in \Pi(V(\lambda))$. 
Since $\cw\cdot \lambda =V(\lambda)$, 
then  $\lambda -\alpha = \lambda - \beta_1-\cdots-  \beta_m \in \Pi(V(\lambda))$, where $-\beta_j \in [T\cup -T] \cap \Phi^-$ for all $1\leq j \leq m$. 
Hence,
$\alpha(h) =  \beta_1(h)+\cdots+  \beta_m(h)$, for all $h\in \mathfrak{h}$. Thus, $\kappa(t_\alpha, h) = \kappa ( t_{\beta_1}+\cdots+  t_{\beta_m}, h)$, for all $h\in \mathfrak{h}$.
By the nondegeneracy of $\kappa$ on $\mathfrak{h}$ [\cite{humphreys}, Corollary 8.2], $t_\alpha =  t_{\beta_1}+\cdots+  t_{\beta_m}$. Therefore,
$\mu(h_\alpha) = \frac{2}{(\alpha,\alpha)} \mu( t_{\beta_1})+\cdots+ \frac{2}{(\alpha,\alpha)} \mu(t_{\beta_m}) =  \frac{2}{(\alpha,\alpha)} \nu( t_{\beta_1})+\cdots+ \frac{2}{(\alpha,\alpha)} \nu(t_{\beta_m}) 
=\nu(h_\alpha)$. 

\vspace{2mm}

\noindent Case $2$. $\lambda(h_\alpha)= 0$: Consider the $\alpha$-strings through $\mu$ and $\nu$:
\begin{equation}
\begin{array}{llllll}
\mu-r\alpha,...,\mu,...,\mu+q\alpha, &  r-q =\langle \mu, \alpha \rangle,& r, q \in \mathbb{N},\\
\nu-r'\alpha,...,\nu,...,\nu+q'\alpha,&  r'-q' =\langle \nu, \alpha \rangle,  &r', q' \in \mathbb{N}.
\end{array}
\end{equation}
If $\langle \mu, \alpha \rangle=\langle \nu, \alpha \rangle$, we're done. 
Otherwise, we must have $r \neq r'$, or $q\neq q'$.  We'll assume $q<q'$, since the other cases are similar. Hence, 
$\nu+\alpha \in \Pi(V(\lambda))$.
Since $[T\cup -T]\cdot \lambda =V(\lambda)$, we have 
\begin{equation}
\begin{array}{rllllll}
\nu&=&\lambda - \gamma_1-\cdots - \gamma_n,\\
\nu+\alpha  &=& \lambda -  \tau_1-\cdots - \tau_p,
\end{array}
\end{equation}
for some $-\gamma_i, -\tau_j \in \cw \cap \Phi^-$ for all $i, j$, with $1\leq i \leq n$, and $1\leq j \leq p$. Hence, $\alpha =  -\tau_1-\cdots - \tau_p + \gamma_1+\cdots + \gamma_n \in \Phi$.  Hence, by Lemma \ref{wideclosure}, $\alpha \in \cw$, a contradiction.  Therefore, we've established $\mu(h_\alpha) =\nu(h_\alpha)$ for all $\alpha \in \Phi$.
Thus, we must have $\mu =\nu$, as required. 
\end{proof}

\begin{lemma}\label{notpandn}
Let $T$ be a closed subset of $\Phi$, and $\mathfrak{t}$  a subalgebra of $\mathfrak{h}$ containing $[\mathfrak{g}_\alpha, \mathfrak{g}_{-\alpha} ]$ for each $\alpha \in T^r$.
Further, let $\wc\cdot \lambda =V(\lambda)$.
Then, for $\mu, -\mu \in \Pi(V(\lambda)^* \otimes V(\lambda))\setminus \{0\}$, we cannot have both $(V(\lambda)^* \otimes V(\lambda))^{\mathfrak{s}_{T, \mathfrak{t}}}_\mu \neq \{0\}$, and $(V(\lambda)^* \otimes V(\lambda))^{\mathfrak{s}_{T, \mathfrak{t}}}_{-\mu} \neq  \{0\}$. 
\end{lemma}
\begin{proof} 
By way of contradiction, suppose that both $(V(\lambda)^* \otimes V(\lambda))^{\mathfrak{s}_{T, \mathfrak{t}}}_\mu \neq \{0\}$, and $(V(\lambda)^* \otimes V(\lambda))^{\mathfrak{s}_{T, \mathfrak{t}}}_{-\mu} \neq \{0\}$. Then, by Lemma \ref{gzero}, 
$\langle \mu, \beta \rangle=0$ for all $\beta \in T$, and hence for all $\beta \in (T\cup -T)$. 

We now establish that, in fact, for all $\beta \in \cw$, we have $\langle \mu, \beta \rangle=0$.  
Since $\beta \in \cw$, we have $\beta= \beta_1+\cdots+\beta_m$, for some $\beta_1,...,\beta_m \in (T\cup-T)$. Hence, 
$(\mu, \beta)=0$ for all $\beta \in \cw$, which implies $\langle \mu, \beta \rangle=0$ for all $\beta \in \cw$. 

Since $\mu \in \Pi(V(\lambda)^* \otimes V(\lambda))$, then $\mu = \tau-\nu$ for some $\tau, \nu \in \Pi(V(\lambda))$. Hence, we have  
$\langle \tau,  \beta \rangle = \langle \nu, \beta \rangle $ for all $\beta \in \cw$. 
By Lemma \ref{greatlemma},  this implies that $\tau=\nu$, so that $\mu=0$, a contradiction. Thus, it must be the case that  not both $(V(\lambda)^* \otimes V(\lambda))^{\mathfrak{s}_{T, \mathfrak{t}}}_\mu \neq \{0\}$, and  $(V(\lambda)^* \otimes V(\lambda))^{\mathfrak{s}_{T, \mathfrak{t}}}_{-\mu} \neq \{0\}$. 
\end{proof}

\begin{lemma}\label{nzero}
Let $T$ be a closed subset of $\Phi$, 
$\lambda \in \Lambda^+$,  $\mathfrak{t}$ be a subalgebra of $\mathfrak{h}$
containing $[\mathfrak{g}_\alpha, \mathfrak{g}_{-\alpha}]$ for every $\alpha \in T^r$, and $\cw \cdot \lambda =V(\lambda)$. Further,  let $V(\rho)$ be a simple $\mathfrak{g}$-module in the
$\mathfrak{g}$-module decomposition of $V(\lambda)^* \otimes V(\lambda)$, with $\rho \neq 0$.
Then,    $V(\rho)^{\mathfrak{s}_{T,\mathfrak{t}}}_0 =V(\rho)^{\mathfrak{s}_{\wc,\mathfrak{h}}}_0   =\{0\}$.  
\end{lemma}
\begin{proof} Observe that $V(\rho)^{\mathfrak{s}_{T,\mathfrak{t}}}_0 = V(\rho)^{\mathfrak{s}_{T,\mathfrak{h}}}_0$. Hence, to show $V(\rho)^{\mathfrak{s}_{T,\mathfrak{t}}}_0  =\{0\}$,
 it suffices to show $V(\rho)^{\mathfrak{s}_{T,\mathfrak{h}}}_0 =\{0\}$.
 By way of contradiction, suppose there is a nonzero $v \in V(\rho)^{\mathfrak{s}_{T, \mathfrak{h}}}_0$. Then, $\mathfrak{g}_\beta \cdot v =0$ for all $\beta \in T$, and $\mathfrak{h} \cdot v =0$. 
 We'll first show that $\mathfrak{s}_{[T \cup-T], \mathfrak{h}} \cdot v =0$, which will imply $V(\rho)^{\mathfrak{s}_{T,\mathfrak{h}}}_0 =V(\rho)^{\mathfrak{s}_{[T\cup -T],\mathfrak{h}}}_0$. 
We establish this in two cases.

\vspace{2mm}

\noindent Case $1$. $\mathit{\beta \in  (T \cup -T)}$: Let  $\alpha \in T$,  and $e_\alpha \in \mathfrak{g}_\alpha$, with corresponding  $\mathfrak{sl}_2$-triple given by
 $e_\alpha \in \mathfrak{g}_\alpha$, $e_{-\alpha} \in \mathfrak{g}_{-\alpha}$ and  $h_\alpha \in [\mathfrak{g}_\alpha, \mathfrak{g}_{-\alpha}]$. Since $e_\alpha \cdot v =0$ and $h_\alpha \cdot v =0$, 
$\mathfrak{sl}_2$-theory implies that $e_{-\alpha} \cdot v =0$, so that $\mathfrak{g}_{-\alpha} \cdot v =0$. Hence, $\mathfrak{g}_{\beta} \cdot v =0$ for all $\beta \in (T \cup -T)$.

\vspace{2mm}

\noindent Case $2$. $\mathit{\beta \in  [T \cup -T]}$:  First we  construct  $[T \cup -T]$ recursively as follows. Define $[T \cup -T]_0 = (T \cup -T)$,
$[T \cup -T]_1 = \{\alpha+\beta \in \Phi \setminus [T \cup -T]_0 ~|~  \alpha, \beta \in [T \cup -T]_0  \}$, and 
$[T \cup -T]_n= \{\alpha+\beta \in \Phi \setminus \cup_{i=0}^{n-1}[T \cup -T]_i ~|~  \alpha, \beta \in \cup_{i=0}^{n-1} [T \cup -T]_{i}  \}$, for $n>1$. Since $\Phi$ is finite,  
there must be a $k \in \mathbb{N}$ such that  $[T \cup -T]_k \neq \emptyset$, but  $[T \cup -T]_{k+1} = \emptyset$. Then $[T \cup -T] = \cup_{i=0}^k [T \cup -T]_i$.

Having constructed $[T \cup -T]= \cup_{i=0}^k [T \cup -T]_i$, we'll now show $\mathfrak{g}_{\beta} \cdot v =0$ for all $\beta \in [T \cup -T]_n$ by induction on $n$. If $\beta \in [T \cup -T]_0 = (T \cup -T)$, then
we've already shown $\mathfrak{g}_{\beta} \cdot v =0$ in Case 1. Assume $\mathfrak{g}_{\tau} \cdot v =0$ for all $\tau \in [T \cup -T]_{i}$, $i \leq n-1$, and consider $\beta \in [T \cup -T]_{n}$. 
We must have $\beta=\gamma+\delta$ for some $\gamma, \delta \in \cup_{i=0}^{n-1} [T \cup -T]_i$. Hence, $\mathfrak{g}_{\gamma+\delta} \cdot v = [\mathfrak{g}_{\gamma}, \mathfrak{g}_{\delta}] \cdot v=  \mathfrak{g}_{\gamma} \cdot \mathfrak{g}_{\delta}\cdot v -  \mathfrak{g}_{\delta}\cdot  \mathfrak{g}_{\gamma} \cdot v =0$, since $\mathfrak{g}_{\gamma} \cdot v =0$ and $\mathfrak{g}_{\delta} \cdot v =0$, by the induction hypothesis.

\vspace{2mm}

Hence, the above two cases establish that $\mathfrak{s}_{[T \cup-T], \mathfrak{h}} \cdot v =0$ (which, as noted above, implies $V(\rho)^{\mathfrak{s}_{T,\mathfrak{h}}}_0 =V(\rho)^{\mathfrak{s}_{[T\cup -T],\mathfrak{h}}}_0$).  We now show that, in fact, $e_\alpha \cdot v=0$, for each $\alpha \in \Phi$. By way of contradiction, suppose 
$e_\alpha \cdot v \neq 0$ for some $\alpha \in \Phi \setminus [T\cup-T]$. Then, $\alpha \in \Pi(V(\lambda)^* \otimes V(\lambda))$ so that $\alpha = \beta_1+\cdots+\beta_n \in \Phi$, for 
some $\beta_1,...,\beta_n \in [T\cup -T]$, since $\cw \cdot \lambda =V(\lambda)$. This implies $\alpha = \beta_1+\cdots+\beta_n \in \cw$, by Lemma \ref{wideclosure}. A contradiction. Hence, it must be the case that
 $e_\alpha \cdot v=0$ for all $\alpha \in \Phi$. However, since $v\in V(\rho)$, and $\rho \neq 0$, this isn't possible. Thus, $V(\lambda)^{\mathfrak{s}_{T,\mathfrak{t}}}_0 =V(\lambda)^{\mathfrak{s}_{T,\mathfrak{h}}}_0 =\{0\}$.
\end{proof}

\begin{lemma}\label{compo}
Let $\mu$ and $\eta  \in \Pi(\End~ V(\lambda))$. Suppose $g_\mu \in (\End~ V(\lambda))_\mu$, and $g_\eta \in (\End ~V(\lambda))_\eta$, then $g_\mu \circ g_\eta \in (\End ~V(\lambda))_{\mu+\eta}$.
\end{lemma}
\begin{proof} 
Consider the weight space decomposition $V(\lambda) = \oplus_i V(\lambda)_{\gamma_i}$, and let $e_{i 1},...,e_{i n_i}$ be a basis of $V(\lambda)_{\gamma_i}$ and $e^{ij}\in V(\lambda)^*$ the corresponding dual of $e_{ij}$. That is, $e^{ij}(e_{kl}) =\delta_{(i,j), (k,l) }$. 

We have $V(\lambda)^* \otimes V(\lambda) \cong   \End ~V(\lambda)$, with the  the familiar isomorphism $\Psi: V(\lambda)^* \otimes V(\lambda)  \rightarrow  \End ~V(\lambda)$, given by
$\Psi (e^{ij} \otimes e_{kl} )(w) = e^{ij}(w) e_{kl}, ~ \text{for}~ w \in V(\lambda)$. Note that  $\Psi(\sum_{ij} e^{ij} \otimes e_{ij} ) =\text{Id}_{V(\lambda)}$.

Observe the weights of basis elements: $\text{wt}( e^{ij})= -\gamma_i$, $\text{wt}( e_{ij})= \gamma_i$, and thus
$\text{wt}( e^{ij} \otimes e_{kl})= -\gamma_i+\gamma_k$. Further,
\begin{equation}
\begin{array}{lllllll}
\Psi ( e^{ij} \otimes e_{kl}) \circ \Psi( e^{mn} \otimes e_{pq}) =
\begin{cases}
\Psi( e^{mn} \otimes e_{kl}), &(i,j)=(p,q),\\
0, &\text{otherwise}.
\end{cases}
\end{array}
\end{equation}
This implies $\text{wt}(\Psi( e^{ij} \otimes e_{kl}) \circ \Psi( e^{mn} \otimes e_{ij})) =\text{wt}( \Psi( e^{mn} \otimes e_{kl}))
= \text{wt} ( \Psi(e^{ij} \otimes e_{kl})) + \text{wt} ( \Psi(e^{mn} \otimes e_{ij}))$. 

Any $g_\mu \in (\End~ V(\lambda))_\mu$ 
may be written as a sum of the image of elementary basis elements $\Psi ( e^{ij} \otimes e_{kl})$, each of weight $\mu$ (i.e., $\mu= -\gamma_i+\gamma_k$).  An 
analogous statement holds for $g_\eta \in (\End~ V(\lambda))_\eta$. Hence, the result follows. 
\end{proof}

The following well-known result will be used to establish indecomposability below. It was used  in \cite{panyu}, for instance.

\begin{lemma}\label{idempotentlemma}
Let  $\mathfrak{s}$ be a subalgebra of the semisimple Lie algebra $\mathfrak{g}$, and $V$ a $\mathfrak{g}$-module.
Then, $V$  is $\mathfrak{s}$-indecomposable if and only if
$(\End ~V)^\mathfrak{s}$ contains no non-trivial idempotents.
\end{lemma}

We are now ready to state and prove our first  result
that  establishes necessary and sufficient conditions for a regular subalgebra to be $\lambda$-wide. The  proof of the theorem has similarity to  that of Theorem $4.2$ in  \cite{panyu}, which considers regular ad-nilpotent subalgebras. The theorem below is more general, and  pertains to $\lambda$-wide regular subalgebras, including those that are solvable, semisimple, or Levi decomposable.
Our proof, in addition, does not employ the consideration of gradings of modules, with the related geometric properties of $\mathcal{Q} \otimes_\mathbb{Z} \mathbb{R}$, 
in contrast to what is done in
 \cite{panyu}. 

\begin{theorem}\label{lwide}
Let $T$ be  a closed subset of $\Phi$, $\mathfrak{t}$  a subalgebra of $\mathfrak{h}$ containing $[\mathfrak{g}_\alpha, \mathfrak{g}_{-\alpha} ]$ for each $\alpha \in T^r$, and $\lambda \in \Lambda^+$. Then, $\mathfrak{s}_{T, \mathfrak{t}}$ is $\lambda$-wide if and only if
$\cw \cdot \lambda =  V(\lambda)$.   
\end{theorem}
\begin{proof}
\noindent $(\Longrightarrow)$ Assume  that $\mathfrak{s}_{T,\mathfrak{t}}$ is $\lambda$-wide. By way of contradiction, suppose that $\wc\cdot \lambda \subsetneq  V(\lambda)$.
Then,  $\wc \cdot \lambda$ is a proper $\mathfrak{s}_{[T \cup -T], \mathfrak{h}}$-submodule of $V(\lambda)$. And, we must have $\wc \subsetneq \Phi$, which follows from [\cite{humphreys}, Theorem 20.2]. Thus, $\mathfrak{s}_{\wc, \mathfrak{h}}$ is a proper subalgebra of $\mathfrak{g}$. Further, it is reductive with radical contained in $\mathfrak{h}$ (i.e., $\mathfrak{s}_{[T \cup -T], \mathfrak{h}} =    \mathfrak{s}_{[T \cup -T], \mathfrak{k}} \oplus \mathfrak{k}^\perp$, where $\mathfrak{k}$ is the subalgebra of $\mathfrak{h}$ generated
by $[\mathfrak{g}_\alpha, \mathfrak{g}_{-\alpha}]$ for all $\alpha \in [T\cup -T] \cap \Phi^+$, and $\mathfrak{k}^\perp =\{h\in \mathfrak{h}~|~ \kappa(h, h')=0, ~\text{for all}~ h'\in \mathfrak{k}\}$). 
This implies that $\mathfrak{s}_{\wc, \mathfrak{h}}$ is completely reducible on $V(\lambda)$ [\cite{dix}, Corollary 1.6.4],  so that
$V(\lambda)$ is $\mathfrak{s}_{\wc, \mathfrak{h}}$-decomposable with $\wc \cdot \lambda$  as a proper component in the decomposition. Hence, since $\mathfrak{s}_{T, \mathfrak{t}} \subseteq \mathfrak{s}_{\wc, \mathfrak{h}}$, then  $V(\lambda)$ is also $\mathfrak{s}_{T, \mathfrak{t}}$-decomposable, a contradiction. Therefore, it must be the case that $\wc \cdot \lambda = V(\lambda)$.

\vspace{2.2mm}

\noindent $(\Longleftarrow)$ Suppose $\cw\cdot \lambda = V(\lambda)$.    
We have the following isomorphism of $\mathfrak{g}$-modules, and decomposition into simple $\mathfrak{g}$-modules:
\begin{equation}
\End~V(\lambda) \cong V(\lambda)^* \otimes V(\lambda) = \bigoplus_{i=0}^k V(\lambda_i), 
\end{equation}
where each $\lambda_i \in \mathcal{Q} \cap \Lambda^+$, and  the multiplicity of $V(0)$ is  one in the weight space decomposition, by Schur's lemma. 
We may assume $\lambda_0=0$, and that $\lambda_i \neq 0$ for $i>0$.  Considering
$V(0) = V(0)^{\mathfrak{s}_{  T, \mathfrak{t} }}$, we then have 
\begin{equation}
( \End ~V(\lambda) )^{\mathfrak{s}_{T, \mathfrak{t}} }\cong  V(0) \oplus \Bigg(\bigoplus_{i=1}^k V(\lambda_i)^{\mathfrak{s}_{T,\mathfrak{t}}} \Bigg),
\end{equation}
where $V(0) = \mathbb{C} \cdot \text{Id}_{V(\lambda)}$. 
Hence, $ g\in (\End ~ V(\lambda))^{\mathfrak{s}_{T, \mathfrak{t}}}$ may be written as
  \begin{equation}\label{summ}
 g= c ~\text{Id}_{V(\lambda)} + \sum_{\mu} g_{\mu},
 \end{equation}
where $c\in \mathbb{C}$, and $g_{\mu} \in  \big(\bigoplus_{i=1}^k V(\lambda_i)^{\mathfrak{s}_{T,\mathfrak{t}}} \big)_{\mu}$.
By Lemma \ref{nzero}, we may assume  that all $\mu \neq 0$. By Lemma \ref{notpandn},  $\mu$ and $-\mu$ may not both occur as  (non-zero) weights in summands within Eq. \eqref{summ}.

Suppose further that $g$ is idempotent. Then
\begin{equation}
c ~\text{Id}_{V(\lambda)} + \sum_{\mu} g_{\mu} = c^2 ~\text{Id}_{V(\lambda)} + 2c \sum_{\mu} g_{\mu} +\Big(\sum_{\mu} g_{\mu} \Big)^2,
\end{equation}
so that 
\begin{equation}\label{ghh}
(c^2-c) ~\text{Id}_{V(\lambda)} + (2c-1) \Big(\sum_{\mu} g_{\mu}\Big) +\Big(\sum_{\mu} g_{\mu} \Big)^2=0.
\end{equation}
Since  it is always the case that $\mu\neq 0$, and that $\mu$ and $-\mu$ may not both occur, then the only term of weight zero in Eq. \eqref{ghh} is $\text{Id}_{V(\lambda)}$. Hence, it must be the case that $c=0$, or $c=1$. Thus
\begin{equation}\label{ghhh}
 \Big(\sum_{\mu} g_{\mu}\Big) \pm\Big(\sum_{\mu} g_{\mu} \Big)^2 = \Big(\sum_{\mu} g_{\mu}\Big) \Big(\text{Id}_{V(\lambda)}\pm \sum_{\mu} g_{\mu}\Big)   =0,
\end{equation}
where the sign $\pm$ depends on whether $c=0$ or $c=1$.

Next, we establish that the factor $\sum_{\mu} g_{\mu} $ in Eq. \eqref{ghhh} is nilpotent. 
By way of contradiction, suppose that $\sum_{\mu} g_{\mu} $ is not nilpotent. Then, there must exist an infinite sequence $\{\nu_1,\nu_2, \nu_3,...\}$ of nonzero  elements of $\Pi( \text{End } ~V(\lambda))\setminus \{0\}$, each occurring in $\sum_{\mu} g_{\mu} $, such that 
$ g_{\nu_k} \circ  g_{\nu_{k-1}} \circ \cdots \circ g_{\nu_1} \neq 0$ for all $k>0$.  This implies that $\nu_k+  \nu_{k-1}+  \cdots + \nu_1 \in \Pi( \text{End } ~V(\lambda)) \setminus \{0\}$ for each $k>0$ (considering Lemmas
 \ref{notpandn}, \ref{nzero}, and \ref{compo}).  
Since  $\Pi( \text{End } ~V(\lambda))$ is finite, there exist $n, m$, with $n<m$ such that $\nu_n+  \nu_{n-1}+  \cdots + \nu_1=\nu_m+  \nu_{m-1}+  \cdots + \nu_1$, so that
$\nu_m+  \nu_{m-1}+  \cdots + \nu_{n+1}=0$. This implies $g_{\nu_m}\circ  g_{\nu_{m-1}}\circ  \cdots \circ g_{\nu_{n+1}}=0$  (again considering Lemmas
 \ref{notpandn}, \ref{nzero}, and \ref{compo}). This in turn
implies $g_{\nu_m}\circ  g_{\nu_{m-1}}\circ  \cdots \circ g_{\nu_{n-1}} \circ g_{\nu_n} \circ \cdots \circ g_{\nu_1}=0$, a contradiction. Hence, it must be the case that
$\sum_{\mu} g_{\mu} $ is nilpotent.

Since $\sum_{\mu} g_{\mu} $ is nilpotent, then $\text{Id}_{V(\lambda)}\pm \sum_{\mu} g_{\mu}$ is invertible.  Eq. \eqref{ghhh} then implies that $\sum_{\mu} g_{\mu}=0$. 
Thus, $g=0$ or $g =\text{Id}_{V(\lambda)}$. Hence, $V(\lambda)$ is $\mathfrak{s}_{T, \mathfrak{t}}$-indecomposable by Lemma \ref{idempotentlemma}.
\end{proof}

We may define a subalgebra to be ${\mathit \lambda}$-{\it narrow} if the simple module of highest weight $\lambda$ has a proper decomposition when restricted to the subalgebra. Then, 
we may restate Theorem \ref{lwide} from the perspective of $\lambda$-narrow subalgebras.

\begin{corollary}
Let $T$ be  a closed subset of $\Phi$, and $\mathfrak{t}$  a subalgebra of $\mathfrak{h}$ containing $[\mathfrak{g}_\alpha, \mathfrak{g}_{-\alpha} ]$ for each $\alpha \in T^r$, and $\lambda \neq 0$.  Then, the regular subalgebra $\mathfrak{s}_{T, \mathfrak{t}}$ is $\lambda$-narrow if and only if $[T \cup -T]\cdot \lambda \subsetneq V(\lambda)$. 
\end{corollary}

Next, we show that Theorem \ref{lwide} implies necessary and sufficient conditions for a regular subalgebra (solvable, semisimple, or Levi decomposable) of a semisimple Lie algebra  to be wide. In doing so, 
we recover the known result for essentially all regular solvable subalgebras to be wide in \cite{panyu}.

\begin{corollary}\label{lwideb}
Let $T$ be  a closed subset of $\Phi$, and $\mathfrak{t}$  a subalgebra of $\mathfrak{h}$ containing $[\mathfrak{g}_\alpha, \mathfrak{g}_{-\alpha} ]$ for each $\alpha \in T^r$.  Then, $\mathfrak{s}_{T, \mathfrak{t}}$ is wide if and only if $[T \cup -T] =\Phi$. 
\end{corollary}
\begin{proof} 
\noindent $(\Longrightarrow)$ Suppose the regular  subalgebra $\mathfrak{s}_{T, \mathfrak{t}}$ is wide. 
By way of contradiction, suppose $[T \cup -T] \subsetneq \Phi$.  Then, there exists $\alpha_i \in \Delta\setminus \cw$. Consider $V(\lambda_i)$, where 
$\lambda_i$ is the fundamental dominant weight (relative to $\Delta$) that is dual to $\alpha_i$. Then, for $e_{-\alpha_i} \in \mathfrak{g}_{-\alpha_i}$,  we have
$e_{-\alpha_i} \cdot v_{\lambda_i} \in V(\lambda_i) \setminus \{0\}$. 

Hence, $\lambda_i -\alpha_i \in \Pi( V(\lambda_i))$. We claim that $\lambda_i -\alpha_i \notin \Pi( \cw \cdot \lambda_i)$. By way of contradiction,
suppose $\lambda_i -\alpha_i \in \Pi( \cw \cdot \lambda_i)$, then $\alpha_i = \beta_1+\cdots +\beta_n \in \Phi$ for some
$\beta_1,..., \beta_n \in \cw$. By Lemma \ref{wideclosure}, this implies $\alpha_i \in \cw$, a contradiction. Hence, it must be the case that $\lambda_i -\alpha_i \notin \Pi( \cw \cdot \lambda_i)$. 

Since $\lambda_i -\alpha_i \in \Pi( V(\lambda_i))$, but  $\lambda_i -\alpha_i \notin \Pi( \cw \cdot \lambda_i)$, then $\cw \cdot V(\lambda_i)$ is a proper
$\sttt$-submodule of $V(\lambda_i)$. Hence, since $\mathfrak{s}_{[T \cup -T], \mathfrak{h}}$ is a  reductive subalgebra of $\mathfrak{g}$ with
radical contained in $\mathfrak{h}$,  $V(\lambda_i)$ has a non-trivial $\sttt$-decomposition  [\cite{dix}, Corollary 1.6.4].  Since, $\mathfrak{s}_{T, \mathfrak{t}} \subseteq \mathfrak{s}_{[T\cup -T], \mathfrak{h}}$,
this is also a non-trivial decomposition with respect to $\stt$. This, however, contradicts our initial assumption that $\stt$ is wide. Hence, it must be the case that 
$\cw=\Phi$.

\vspace{2.2mm}

\noindent $(\Longleftarrow)$ Suppose  $[T \cup -T]=\Phi$.
Then, for all $\lambda \in \Lambda^+$, we have $\wc \cdot \lambda =V(\lambda)$, which follows from [\cite{humphreys}, Theorem 20.2b]. Hence, by Theorem \ref{lwide}, $\mathfrak{s}_{T, \mathfrak{t}}$ is 
$\lambda$-wide for all $\lambda\in \Lambda^+$. Hence, $\mathfrak{s}_{T, \mathfrak{t}}$ is wide.
\end{proof}

 In our final result of this section that follows, we assume that $\mathfrak{g}$ is simple so that the adjoint representation is simple.  We show that establishing whether or not a regular subalgebra  of  a simple Lie algebra is wide does not require consideration of all simple modules.  It is sufficient to only consider the adjoint representation.

\begin{corollary}\label{best} 
The regular subalgebra $\mathfrak{s}_{T, \mathfrak{t}}$ of  a simple Lie algebra $\mathfrak{g}$ is  wide if and only if the adjoint representation of $\mathfrak{g}$ 
is indecomposable when restricted to $\mathfrak{s}_{T, \mathfrak{t}}$. 
\end{corollary}
\begin{proof}
\noindent $(\Longrightarrow)$ If $\mathfrak{s}_{T, \mathfrak{t}}$ is wide in $\mathfrak{g}$, then all simple modules of $\mathfrak{g}$ remain indecomposable when
restricted to $\mathfrak{s}_{T, \mathfrak{t}}$, including the adjoint representation.

\noindent $(\Longleftarrow)$ Assume that the adjoint representation of $\mathfrak{g}$ 
is indecomposable when restricted to $\mathfrak{s}_{T, \mathfrak{t}}$. By way of contradiction, suppose that $\mathfrak{s}_{T, \mathfrak{t}}$ is not wide, so that
$[T \cup -T] \subsetneq \Phi$ (Corollary \ref{lwideb}). Then, the regular reductive subalgebra $\mathfrak{s}_{[T\cup-T], \mathfrak{h}}$, with radical contained in $\mathfrak{h}$,  is a proper subalgebra of $\mathfrak{g}$.  By [\cite{dix}, Corollary 1.6.4], the adjoint representation is completely reducible with respect to 
$\mathfrak{s}_{[T\cup-T], \mathfrak{h}}$. Hence, since $\mathfrak{s}_{[T\cup-T], \mathfrak{h}}$ is a proper $\mathfrak{s}_{[T\cup-T], \mathfrak{h}}$-submodule of the 
adjoint representation of $\mathfrak{g}$, the adjoint representation of $\mathfrak{g}$ has a non-trivial $\mathfrak{s}_{[T\cup-T], \mathfrak{h}}$-decomposition. Since 
$\mathfrak{s}_{T, \mathfrak{t}} \subseteq \mathfrak{s}_{[T\cup-T], \mathfrak{h}}$, this is also a non-trivial $\mathfrak{s}_{T, \mathfrak{t}}$-decomposition, a contradiction.
Hence, it must be the case that $\mathfrak{s}_{T, \mathfrak{t}}$ is wide.
\end{proof}

\section{Regular subalgebras of $\mathfrak{sl}_{n+1}$}\label{narrowwidesection}

In this section, we show that a regular subalgebra of $\mathfrak{sl}_{n+1}$ must be narrow or wide. This is a property that, as discussed above in Section \ref{intro}, does not hold in general for non-regular subalgebras of $\mathfrak{sl}_{n+1}$; a non-regular subalgebra of $\mathfrak{sl}_{n+1}$ may be neither narrow nor wide.
We first introduce  convenient bases for $\mathfrak{sl}_{n+1}$-modules created by Feigin,  Fourier, and Littelmann \cite{feigin}.

The special linear algebra   $\mathfrak{sl}_{n+1}$ is a simple Lie algebra of rank $n$. It has simple roots $\Delta= \{ \alpha_1,..., \alpha_n\}$, and  positive roots 
\begin{equation}
\Phi^+= \{ \alpha_{p, q} \coloneqq \alpha_p+\cdots + \alpha_q~|~ 1 \leq p \leq q \leq n\}.
\end{equation}
Note that $\alpha_i =\alpha_{i,i}$. For ease of notation, define $f_{\beta} \coloneqq e_{-\beta}$ for $\beta \in \Phi^+$.

Before we describe a basis of simple $\mathfrak{sl}_{n+1}$-modules, we first introduce  Dyck paths defined in \cite{feigin}. A {\it Dyck path} is  a sequence of positive roots
\begin{equation}
\mathbf{p} = ( \beta(0), \beta(1), ... , \beta(k)), ~ k\geq 0,
\end{equation} 
satisfying the following conditions:
\begin{enumerate}[i.]
\item If $k=0$, then $\mathbf{p}$ is of the form $\mathbf{p}=(\alpha_i)$ for some simple root $\alpha_i \in \Delta$.
\item If $k\geq 1$, then
\begin{enumerate}[a.]
\item $\beta(0) =\alpha_i$ and $\beta(k) =\alpha_j$ for some $1 \leq i< j \leq n$; and
\item if $\beta(s) = \alpha_{p,q}$, then $\beta(s+1) = \alpha_{p+1,q}$, or $\beta(s+1) = \alpha_{p,q+1}$. 
\end{enumerate}
\end{enumerate}

For any multi-exponent $\mathbf{s} =(s_\beta)_{\beta \in \Phi^+}$, $s_\beta \in \mathbb{N}$, fix an arbitrary order of factors $f_\beta$ in the product
$\Pi_{\beta \in \Phi^+} f^{s_\beta}_\beta$. Let $f^{\mathbf{s}}$ be the order product
\begin{equation}
f^{\mathbf{s}} = \Pi_{\beta \in \Phi^+} f^{s_\beta}_\beta,
\end{equation}
in the universal enveloping algebra of the negative root space of $\mathfrak{sl}_{n+1}$.  We are now ready to present bases for simple $\mathfrak{sl}_{n+1}$-modules.

\begin{theorem} \cite{feigin}\label{fflv}
Let $\lambda= m_1 \lambda_1+\cdots + m_n \lambda_n$ be a dominant $\mathfrak{sl}_{n+1}$-weight and let $S(\lambda)$ be the set of all multi-exponents $\mathbf{s}= (s_\beta)_{\beta \in \Phi^+}$, $s_\beta \in \mathbb{N}$,
such that for all  Dyck paths $\mathbf{p} = ( \beta(0), \beta(1), ... , \beta(k))$ 
\begin{equation}\label{ff1}
s_{\beta(0)}+\cdots +s_{\beta(k)} \leq m_i+\cdots +m_j, 
\end{equation}
where $\beta(0)=\alpha_i$ and $\beta(k)=\alpha_j$. 
Then, the set $f^{\mathbf{s}} v_\lambda$ with $\mathbf{s} \in S(\lambda)$ forms a basis of $V(\lambda)$, which we'll  denote $\mathcal{B}_\lambda$.
\end{theorem}

We'll state one last lemma before we get to our final result showing that  a regular subalgebra of $\mathfrak{sl}_{n+1}$ must be either narrow or wide.

\begin{lemma}\label{nonzero}
Let $V(\lambda)$ be the simple $\mathfrak{sl}_{n+1}$-module of highest weight $\lambda =m_1\lambda_1+\cdots  +m_n\lambda_n$. 
Fix $i$ such that $1\leq i \leq n$, and such that $m_i>0$. Then, for any $j', j$,  such that $1\leq  j' \leq i \leq j \leq n$, we have
 $f_{\alpha_{i,j}} v_\lambda \in \mathcal{B}_\lambda$, and $f_{\alpha_{j',i}} v_\lambda \in \mathcal{B}_\lambda$.
\end{lemma}
\begin{proof}
First note that $m_i \geq 1$, since $m_i >0$.
We'll  only show that $f_{\alpha_{i,j}} v_\lambda \in \mathcal{B}_\lambda$, since showing $f_{\alpha_{j',i}} v_\lambda \in \mathcal{B}_\lambda$ is handled in a similar manner.  

Let $\mathbf{s}= (s_\beta)_{\beta \in \Phi^+}$  be the multi-exponent corresponding to
$f_{\alpha_{i,j}} v_\lambda$. In particular,  $s_{\alpha_{i,j}}=1$, and $s_\beta =0$ if $\beta \in \Phi^+ \setminus \{\alpha_{i,j}\}$. And let
 $\mathbf{p}= ( \beta(0), \beta(1), ... , \beta(k))$ be a  Dyck path, with $p(0) =\alpha_l$, and $p(k) = \alpha_{l'}$, where $1\leq l \leq l' \leq n$. We must show that $\mathbf{p}$ satisfies Eq. \eqref{ff1}   of Theorem \ref{fflv}. 
 
 If  $\alpha_{i,j}$ is not a term in  $\mathbf{p}$, then $s_{\beta(0)}+\cdots +s_{\beta(k)} =0$. Hence,  
 Eq. \eqref{ff1} is satisfied.
If the   Dyck path $\mathbf{p}$ includes $\alpha_{i,j}$ as a term, then  $s_{\beta(0)}+\cdots +s_{\beta(k)} =1$. Further, in this case
 then $l \leq i \leq j  \leq l'$. Thus, considering that $m_i \geq 1$, 
$m_l+\cdots+m_{l'}\geq 1$. Hence, Eq. \eqref{ff1}  of Theorem \ref{fflv} is satisfied, as required.
\end{proof}

\begin{theorem}
A  regular subalgebra of $\mathfrak{sl}_{n+1}$ must be either narrow or wide.
\end{theorem}
\begin{proof}
Let $\stt$ be a regular subalgebra of $\mathfrak{sl}_{n+1}$. Suppose that $\stt$ is not wide. We must show that $\stt$ is narrow.
That is, given an arbitrary simple $\mathfrak{sl}_{n+1}$-module $V(\lambda)$ with $\lambda =m_1\lambda_1+\cdots  +m_n\lambda_n  \neq 0$,  we'll
show that $V(\lambda)$ has a non-trivial $\stt$-decomposition.
Since $\stt$ is not wide, $[T \cup -T] \subsetneq \Phi$ by Corollary \ref{lwideb}. Hence,
$\Delta \setminus [T \cup -T]$ is not empty.  

Since $\lambda \neq 0$, there exists $i$ such that $m_i >0$.
By Lemma \ref{nonzero}, we have
\begin{equation}
f_{\alpha_{i,j}} v_\lambda \in \mathcal{B}_\lambda, ~\text{and}~ f_{\alpha_{j',i}} v_\lambda \in \mathcal{B}_\lambda.
\end{equation}
for each $j, j'$ such that $1\leq j'\leq i \leq j \leq n$.

We proceed in cases, depending on the relationship between $\alpha_i$ and $\cw$, to show that $V(\lambda)$ has a non-trivial $\stt$-decomposition.

\vspace{2mm}

\noindent Case $1$. $\alpha_i \in \cw$: Then, there exists $j'$ or $j$, with $1\leq j'  <i<  j   \leq n$, such that $\alpha_{j'} \notin \cw$, or $\alpha_{j} \notin \cw$.
Take $j' < i$ maximal such that $\alpha_{j'}  \notin \cw$ (if such a $j'$ exists). And,  take $j > i$ minimal such that $\alpha_{j}  \notin \cw$ (if such a $j$ exists). 

If there exists $j' < i$ maximal such that $\alpha_{j'}  \notin \cw$, then $f_{\alpha_{j',i}} v_\lambda \in \mathcal{B}_\lambda$ (Lemma \ref{nonzero}). 
Thus, 
 $\lambda-\alpha_{j'}-\cdots-\alpha_i \in \Pi(V(\lambda))$. We claim that $\lambda-\alpha_{j'}-\cdots-\alpha_i  \notin \Pi(\cw \cdot \lambda)$. By way 
of contradiction, suppose that $\lambda-\alpha_{j'}-\cdots-\alpha_i \in \Pi(\cw \cdot \lambda)$. 

Then,
$\lambda-\alpha_{j'}-\cdots-\alpha_i  = \lambda -\beta_1-\cdots -\beta_k$ for some  $-\beta_1$,...,$-\beta_k \in \cw \cap \Phi^-$. Hence,
$\alpha_{j'} =\beta_1+\cdots +\beta_k-\alpha_{j'+1}-\cdots-\alpha_{i}  \in \Phi$, which is thus an element of $\cw$ by Lemma \ref{wideclosure} (note that
$\beta_1,..., \beta_k, \alpha_{j'+1},...,\alpha_{i}  \in \cw$), a contradiction. 

Hence, it must be the case that $\lambda-\alpha_{j'}-\cdots-\alpha_i  \notin \Pi(\cw \cdot \lambda)$. Since $\lambda-\alpha_{j'}-\cdots-\alpha_i  \notin \Pi(\cw \cdot \lambda)$ and
$\lambda-\alpha_{j'}-\cdots-\alpha_i  \in \Pi(V(\lambda))$, Theorem \ref{lwide} implies that $V(\lambda)$ has a non-trivial $\stt$-decomposition.

The sub-case in which  there exists  $j > i$ minimal    such that $\alpha_{j}  \notin \cw$ follows in a similar manner to the above sub-case. Hence, we omit this sub-case.

\vspace{2mm}

\noindent Case $2$. $\alpha_i \notin \cw$: In this case, $f_{\alpha_i} v_\lambda \in \mathcal{B}_\lambda$ (Lemma \ref{nonzero}). 
Hence, $\lambda -\alpha_i \in \Pi(V(\lambda))$.  As in the above case, we'll show that 
 $\lambda-\alpha_i \notin \Pi(\cw \cdot \lambda)$, which will imply that $V(\lambda)$ has a non-trivial $\stt$-decomposition (Theorem \ref{lwide}).
 
 By way 
of contradiction, suppose that $\lambda-\alpha_i \in \Pi(\cw \cdot \lambda)$. 
Then 
$\lambda-\alpha_i  = \lambda -\beta_1-\cdots -\beta_k$ for some $-\beta_1$,...,$-\beta_k \in \cw \cap \Phi^-$. Therefore
$\alpha_i =\beta_1+\cdots +\beta_k \in \Phi$, which is thus an element of $\cw$ by Lemma \ref{wideclosure}, a contradiction. Hence, it must be the case that $\lambda-\alpha_i \notin \Pi(\cw \cdot \lambda)$, as required. 

Cases $1$ and $2$ are exhaustive since  $\Delta \setminus [T \cup -T]$ is not empty. Hence, we've established that if $\stt$ is not wide, then $\stt$ is narrow, as required.
\end{proof}

\end{document}